\documentclass[11pt]{amsart}

\usepackage{amssymb}
\usepackage{amsmath}
\usepackage{mathrsfs}
\usepackage{dsfont}

\newtheorem{theorem}{Theorem}
\newtheorem{lemma}{Lemma}
\newtheorem{proposition}{Proposition}
\newcommand{\abs}[1]{\lvert#1\rvert}

\begin{document}

\title[Zeros of partial sums of the Dedekind zeta function of a cyclotomic field]{Zeros of partial sums of the Dedekind zeta function of a cyclotomic field}

\author[Andrew Ledoan]{Andrew Ledoan}

\address{Department of Mathematics, University of Tennessee at Chattanooga, 417F EMCS Building (Department 6956), 615 McCallie Avenue, Chattanooga, Tennessee 37403-2598, United States of America}

\email{andrew-ledoan@utc.edu}

\author[Arindam Roy]{Arindam Roy}

\address{Department of Mathematics, University of Illinois at Urbana-Champaign, 165 Altgeld Hall, 1409 W. Green Street (MC-382), Urbana, Illinois 61801-2975, United States of America}

\email{roy22@illinois.edu}

\author[Alexandru Zaharescu]{Alexandru Zaharescu}

\address{Department of Mathematics, University of Illinois at Urbana-Champaign, 449 Altgeld Hall, 1409 W. Green Street (MC-382), Urbana, Illinois 61801-2975, United States of America}

\email{zaharesc@illinois.edu}



\begin{abstract}

In this article, we study the zeros of the partial sums of the Dedekind zeta function of a cyclotomic field $K$ defined by the truncated Dirichlet series
\[
\zeta_{K, X} (s)
 = \sum_{ \|\mathfrak{a}\| \leq  X } \frac{1}{\|\mathfrak{a}\|^{s}},
\]
where the sum is to be taken over nonzero integral ideals $\mathfrak{a}$ of $K$ and $\|\mathfrak{a}\|$ denotes the absolute norm of $\mathfrak{a}$. Specifically, we establish the zero-free regions for $\zeta_{K, X} (s)$ and estimate the number of zeros of $\zeta_{K, X} (s)$ up to height $T$.
\end{abstract}

\maketitle

\thispagestyle{empty}

\section{Introduction and statement of results}

A first generalization of the Riemann zeta-function $\zeta (s)$ is provided by the Dirichlet $L$-functions. Subsequently, Dedekind studied the zeta function $\zeta_{K} (s)$ of an arbitrary algebraic number field $K$, defined for $\mbox{Re} (s) > 1$ by
\begin{equation*}
\zeta_{K} (s)
 = \sum_{\mathfrak{a}} \frac{1}{\|\mathfrak{a}\|^{s}}
 = \sum_{n = 1}^{\infty} \frac{a (n)}{n^{s}},
\end{equation*}
where the first sum is to be taken over all nonzero integral ideals $\mathfrak{a}$ of $K$ and where $\|\mathfrak{a}\|$ denotes the absolute norm of $\mathfrak{a}$. In the second sum, $a (n)$ is used to denote the number of integral ideals $\mathfrak{a}$ with norm $\|\mathfrak{a}\| = n$. 

As in the particular case $K = \mathds{Q}$, where $\zeta (s) = \zeta_{\mathds{Q}} (s)$, the function $\zeta_{K} (s)$ is analytic everywhere except solely for a simple pole at $s = 1$. (See Davenport \cite{Davenport2000} and Neukrich \cite{Neukirch1999}.) The residue of this pole is given by the analytic class number formula
\begin{equation*}
\mathop{\mathrm{Res}}_{s = 1} \ \! (\zeta_{K} (s))
  = \frac{2^{r} \pi^{n_0 - r} R_{K} h_{K}}{w_{K} \sqrt{\abs{d_{K}}}},
\end{equation*}
where $r = r_{1} + r_{2}$ (with $r_{1}$ being the number of real embeddings and $r_{2}$ being the number of complex conjugate pairs of complex embeddings of $K$), $n_{0} = \lbrack K \colon \mathds{Q} \rbrack$ denotes the degree of $K / \mathds{Q}$, $R_{K}$ denotes the regulator, $h_{K}$ denotes the class number, $w_{K}$ denotes the number of roots of unity in $K$, and $d_{K}$ denotes the discriminant of $K$. (See Neukrich \cite[page 467]{Neukirch1999}.)

For $\zeta (s)$, Hardy and Littlewood \cite{HardyLittlewood} provided the approximate functional equation
\begin{equation*}
\zeta(s)
 = \sum_{n \leq X} \frac{1}{n^{s}} + \pi^{s - 1 / 2} \frac{\Gamma((1 - s) / 2)}{\Gamma(s / 2)} \sum_{n \leq Y} \frac{1}{n^{1 - s}} + O (X^{-\sigma}) + O (Y^{\sigma - 1} \abs{t}^{-\sigma + 1 / 2}),
\end{equation*}
where $s = \sigma + it$, $0 \leq \sigma \leq 1$, $X > H > 0,$ $Y > H > 0,$ and $2 \pi X Y = \abs{t}$, with the constant implied by the big-$O$ term depending on $H$ only. Such approximate functional equations motivate the study of properties of the partial sums $F_{X} (s)$ of $\zeta (s)$ defined by
\begin{equation*}
F_{X} (s)
 \mathrel{\mathop:}= \sum_{n \leq  X } \frac{1}{n^{s}}.
\end{equation*}

Gonek and one of the authors \cite{GonekLedoan2010} studied the distribution of zeros of the partial sums $F_{X} (s)$. The authors denote the number of typical zeros $\rho_{X} = \beta_{X} + i \gamma_{X}$ of the partial sums $F_{X} (s)$ with ordinates $0 \leq \gamma_{X} \leq T$ by $N_{X} (T)$. In the case that $T$ is the ordinate of a zero, they define $N_{X} (T)$ as $\lim_{\epsilon \to 0^{+}} N_{X} (T + \epsilon)$. In \cite{GonekLedoan2010}, the authors are concerned with results on $N_X(T)$ as both $X$ and $T$ tends to infinity. 

Theorem 1 in \cite{GonekLedoan2010} collects together a number of known results on the zeros of $F_{X} (s)$ (see Borwein, Fee, Ferguson and Waal  \cite{BorweinFeeFergusonWaall2007}, Montgomery \cite{Montgomery1983}, and Montgomery and Vaughan \cite{MontgomeryVaughan2001}), which can be summarized as follows: 

\bigskip
{\it The zeros of $F_{X} (s)$ lie in the strip $\alpha < \sigma < \beta$, where $\alpha$ and $\beta$ are the unique solutions of the equations $1 + 2^{-\sigma} + \dots + (X - 1)^{-\sigma} = X^{-\sigma}$ and $2^{-\sigma} + 3^{-\sigma} + \dots + X^{-\sigma} = 1$, respectively. In particular, $\alpha > -X$ and $\beta < 1.72865$. Furthermore, there exists a number $X_{0}$ such that if $X \geq X_{0}$, then $F_{X} (s)$ has no zeros in the half-plane
\begin{equation*}
\sigma
 \geq 1 + \left(\frac{4}{\pi} -1 \right) \frac{\log\log X}{\log X}.
\end{equation*}
On the other hand, for any constant $C$ satisfying the inequalities $0 < C < 4 / \pi - 1$ there exists a number $X_{0}$ depending on $C$ only such that if $X \geq X_0$, then $F_{X} (s)$ has  zeros in the half-plane
\begin{equation*}
\sigma
 > 1 + \frac{C \log\log X}{\log X}.
\end{equation*}}

\medskip
Theorem 2 in \cite{GonekLedoan2010} (see also Langer \cite{Langer1931}) can be summarized as follows: 

\bigskip
{\it If $X$ and $T$ are both greater than or equal to 2, then one has
\begin{equation*}
\left|N_{X} (T) - \frac{T}{2 \pi} \log \lbrack X \rbrack\right|
 < \frac{X}{2}.
\end{equation*}}

Here and henceforth, $\lbrack X \rbrack$ denotes the greatest integer less than or equal to $X$. Chandrasekharan and Narasimhan \cite{ChandrasekharanNarasimhan1963} gave an approximate functional equation for the Dedekind zeta function
\begin{equation} \label{apfe}
\zeta_{K} (s)
 = \sum_{n \leq X} \frac{a (n)}{n^{s}} + B^{2 s - 1}\frac{A(1 - s)}{A (s)} \sum_{n \leq Y} \frac{a (n)}{n^{1 - s}} + O (X^{1 - \sigma - 1 / n_{0}} \log X),
\end{equation}
where $A (s) = \Gamma^{r_{1}} (s / 2) \Gamma^{r_{2}} (s)$, $B = 2^{r_{2}} \pi^{n_{0} / 2} / \sqrt{\abs{d_{K}}}$, $X > H > 0,$ $Y > H > 0,$  $X Y = \abs{d_{K}} (\abs{t} / 2 \pi)^{n_{0}}$, and $C_{1} < X / Y < C_{2}$ for some  constants $C_{1}$ and $C_{2}$. In the present article, we investigate the distribution of zeros of the partial sums of the function $\zeta_{K} (s)$ defined by
\begin{equation*}
\zeta_{K, X} (s)
 \mathrel{\mathop:}= \sum_{\|\mathfrak{a}\| \leq X} \frac{1}{\|\mathfrak{a}\|^{s}}
 = \sum_{n \leq X} \frac{a (n)}{n^{s}},
\end{equation*}
which appears in the approximate functional equation \eqref{apfe}. Our purpose is to determine whether the partial sums $\zeta_{K, X} (s)$ exhibit similar properties. To this end, we denote the number of non-real zeros $\rho_{K, X} = \beta_{K, X} + i \gamma_{K, X}$ of the partial sums $\zeta_{K, X} (s)$ with ordinates $0 \leq \gamma_{K, X} \leq T$ by $N_{K, X} (T)$. If $T$ is the ordinate of a zero, then $N_{K, X} (T)$ is to be defined by $\lim_{\epsilon \to 0^{+}} N_{K, X} (T + \epsilon)$.

Our first result about the zeros of  $\zeta_{K, X} (s)$ is summarized as follows.

\begin{proposition} \label{prop1}
Let $K$ be an arbitrary algebraic number field of degree $n_{0} = \lbrack K \colon \mathds{Q} \rbrack$ over the field $\mathds{Q}$ of rational numbers, let $X$ be a real number greater than or equal to 2, and denote by $s$ the complex variable $\sigma + i t$. Then there exist two real numbers $\alpha$ and $\beta$, with $\alpha$ depending on $n_{0}$ and $X$ only and with $\beta$ depending on $n_{0}$ only, such that the zeros of the partial sums $\zeta_{K, X} (s)$ all lie within the rectilinear strip of the complex plane given by the inequalities $\alpha < \sigma < \beta$.  
\end{proposition}

Our second theorem provides an approximate formula for $N_{K, X} (T)$, the number of zeros of the partial sums $\zeta_{K, X} (s)$ in the rectangle determined by the inequalities $\alpha < \sigma < \beta$ and $0 < t < T$, where $\alpha$ and $\beta$ are provided in Proposition \ref{prop1}. Let $K$ be any algebraic number field of degree $n_{0} = \lbrack K \colon \mathds{Q} \rbrack$ over the field $\mathds{Q}$ of rational numbers. In a similar fashion to the case of Riemann zeta function (see \cite{GonekLedoan2010} and \cite{Langer1931}), it can be shown that
\begin{equation} \label{lrz2}
\left\lvert N_{K, X} (T)-\frac{T}{2 \pi} \log \lbrack X \rbrack\right\rvert\leq \frac{X}{2},
\end{equation}
where $X$ and $T$ both go to infinity together. However, if $K = \mathds{Q} (\zeta_{q})$ is a cyclotomic field, where $q\geq 2$, we can significantly improve the error term in \eqref{lrz2}.

\begin{theorem} \label{thm1}
Let $q \geq 2$, let ${\zeta_{q}}$ be a primitive root of unity of order $q$, let $K = \mathds{Q}(\zeta_{q})$, and let $T, X \geq 3$. Let, further, $N$ be the largest integer less than or equal to $X$ such that $a (N)\neq 0$. We have
\begin{equation} \label{lrz1}
N_{K, X} (T)
 = \frac{T}{2 \pi} \log N + O_{q} \left(X \left(\frac{\log\log X}{\log X}\right)^{1 - 1 / \phi (q)}\right),
\end{equation}
where $\phi$ is Euler's totient function.
\end{theorem}

\section{Preliminary Results}
To prove Theorem \ref{thm1}, we will make use of two auxiliary lemmas.
\begin{lemma} \label{l1}
Fix a positive integer $q\geq2$. We have
\begin{equation*}
\#\{n \leq y \colon \mu (n) \neq 0\ \!\text{and }\ \! p\! \mid\! n\ \!\text{ imply }\ \! p\equiv 1\pmod{q}\}
 = O_{q} \left(y \left(\frac{\log\log y}{\log y}\right)^{1 - 1 / \phi (q)}\right),
\end{equation*}
where $\mu$ denotes the M\"{o}bius function. 
\end{lemma}

\begin{proof}
Fix a positive integer $q\geq 2$ and define
\begin{equation*}
\mathcal{B} (q, y)
 \mathrel{\mathop:}= \{n \leq y \colon \mu (n)\neq 0\ \!\text{and }\ \! p\! \mid\! n\ \!\text{ imply }\ \! p\equiv 1\pmod{q}\}.
\end{equation*}
We apply Brun's pure sieve to estimate the size of the set $\mathcal{B}(q,y)$. (See Murty and Cojocaru \cite[page 86]{RammurtyCojocaru2006}.) Let $\mathcal{A}$ be the set of all positive integers $n\leq y$. Let $\mathcal{P}$ be the set of all primes $p$ incongruent to $1$ modulo $q$. Let $\mathcal{A}_p$ be the set of elements of $\mathcal{A}$ which are divisible by $p$. Let, further, $\mathcal{A}_1 \mathrel{\mathop:}= \mathcal{A}$ and $\mathcal{A}_d \mathrel{\mathop:}= \bigcap_{p \mid d}\mathcal{A}_p$, where $d$ is a square-free positive integer composed of a list of prime factors from $\mathcal{P}$.  For any positive real number $z$, we define
\begin{equation*}
S(\mathcal{A}, \mathcal{P}, z)
 \mathrel{\mathop:}= \mathcal{A} \setminus \bigcup_{p \mid P (z)} \mathcal{A}_{p},
\end{equation*}
where
\begin{equation*}
P (z)
 \mathrel{\mathop:}= \prod_{\substack{p \in \mathcal{P} \\ p < z}} p.
\end{equation*}

We consider the multiplicative function $\omega$ defined for all primes $p$ by $\omega (p) \mathrel{\mathop:}= 1$. We have 
\begin{equation*}
\#\mathcal{A}_{d}
 = \#\{n \leq y \colon n \equiv 0 \pmod{d}\}
 = \frac{\omega (d)}{d} y + R_{d},
\end{equation*}
where
\begin{equation*}
\abs{R_{d}}
 \leq \omega (d).
\end{equation*} 
From Mertens's estimates, we have
\begin{equation*}
\sum_{\substack{p \in \mathcal{P} \\ p < z}} \frac{\omega (p)}{p}
 = \frac{\phi (q) - 1}{\phi (q)} \log \log z + O (1).
\end{equation*}
For the sake of brevity, we let 
\begin{equation*}
W(z)
 \mathrel{\mathop:}= \prod_{p \mid P (z)} \left(1 - \frac{\omega (p)}{p}\right).
\end{equation*}
By Brun's pure sieve, we have 
\begin{equation} \label{sv1}
\#S (\mathcal{A}, \mathcal{P},z)
 = y W (z) \left(1 + O \left((\log z)^{-A}\right)\right) + O \left(z^{\eta \log \log z}\right),
\end{equation}
where $A = \eta \log \eta$ and, for some $\alpha < 1$,
\begin{equation*}
\eta
 = \frac{\alpha \log y}{\log z \log \log z}.
\end{equation*}

Since $\omega (p)=1$, Mertens's estimates yield 
\begin{equation} \label{me2}
W(z)
 = O \left(\frac{1}{(\log z)^{1 - 1 / \phi (q)}}\right).
\end{equation}
We now choose $\log z = c \log y / \log\log y$. Then for a suitable positive and sufficiently small constant $c$ and from \eqref{sv1} and \eqref{me2}, we have
\begin{equation} \label{fe}
\#S (\mathcal{A}, \mathcal{P}, z)
 = O \left(y \left(\frac{\log \log y}{\log y}\right)^{1 - 1 / \phi (q)}\right).
\end{equation}
Since $\mathcal{B} (q, y) \subseteq S (\mathcal{A}, \mathcal{P}, z)$, we have $\#\mathcal{B} (q, z)\leq \#S (\mathcal{A}, \mathcal{P}, z)$. Employing this last inequality together with \eqref{fe}, we complete the proof of Lemma \ref{l1}.
\end{proof}

\begin{lemma} \label{l2}
Let $q \geq 2$ and let $K = \mathds{Q} (\zeta_{q})$. Let, further, 
\begin{equation*}
\zeta_{K} (s)
 = \sum_{n = 1}^{\infty} \frac{a (n)}{n^{s}}.
\end{equation*}
We have
\begin{equation*}
\#\{n \leq x \colon a (n) \neq 0\}
 = O_{q} \left(x \left(\frac{\log \log x}{\log x}\right)^{1 - 1 / \phi (q)}\right).
\end{equation*}
\end{lemma}

\begin{proof}
Let $K = \mathds{Q} (\zeta_{q})$, where $\zeta_{q}$ is a primitive root of unity of order $q$.
We have
\begin{equation*}
\zeta_{K}
 = \prod_{\mathcal{P} \mid q} \left(1 - \frac{1}{\| \mathcal{P} \|^s}\right)^{-1} F_{q} (s),
\end{equation*}
where
\begin{equation*}
F_{q} (s)
 = \prod_{\chi \pmod{q}} L (s, \chi).
\end{equation*}
(See \cite[page 468]{Neukirch1999}.) For $\sigma > 1$, we have
\begin{equation*}
F_{q} (s)
 = \prod_{\chi \pmod{q}} \prod_{\substack{p \ \! \text{prime} \\ p \nmid q}} \left(1 - \frac{\chi (p)}{p^{s}}\right).
\end{equation*}
Hence, for $\sigma > 1$, we have
\begin{equation*}
\begin{split}
\log F_{q} (s)
 &= -\sum_{\chi \pmod{q}} \sum_{\substack{p \ \! \text{prime} \\ p \nmid q}} \log \left(1 - \frac{\chi (p)}{p^{s}}\right) \\
 &= \sum_{\chi \pmod{q}} \sum_{\substack{p \ \! \text{prime} \\ p \nmid q}} \sum_{m = 1}^{\infty} \frac{\chi (p)}{m p^{m s}} \\
 &= \sum_{\substack{p \ \! \text{prime} \\ p \nmid q}} \sum_{m = 1}^{\infty} \sum_{\chi \pmod{q}} \chi (p^{m}),
\end{split}
\end{equation*}
where
\begin{equation*}
\sum_{\chi \pmod{q}} \chi (p^{m})
 = \left\{ \begin{array}{ll}
      \phi (q), & \mbox{if $p^{m} \equiv 1 \pmod{q}$;} \\
      0, & \mbox{otherwise.} \\
\end{array}
\right.
\end{equation*}
It follows that
\begin{equation*}
\log F_{q} (s)
 = \sum_{\substack{p \ \! \text{prime}, \ \! m \geq 1 \\ p^{m} \equiv 1\pmod{q}}} \frac{\phi (q)}{m p^{m s}}.
\end{equation*}
Hence, we have
\begin{equation*}
F_{q} (s)
 = \exp \left(\sum_{\substack{p \ \! \text{prime},\ \! m \geq 1 \\ p^{m} \equiv 1\pmod{q}}} \frac{\phi (q)}{m p^{m s}}\right).
\end{equation*}

Now, for $\sigma > 1$,
\begin{equation*}
F_{q} (s)= \sum_{n = 1}^{\infty} \frac{c (n)}{n^{s}}
 = \prod_{p \ \! \text{prime}} \left(1 + \frac{c (p)}{p^{s}} + \frac{c (p^{2})}{p^{2 s}} + \ldots\right).
\end{equation*}
Thus, we have
\begin{equation*}
\begin{split}
\log F_{q} (s)
 &= \sum_{p \ \! \text{prime}} \log \left(1 + \frac{c (p)}{p^{s}} + \frac{c (p^{2})}{p^{2 s}} + \ldots\right) \\
 &= \sum_{p \ \! \text{prime}} \sum_{m = 1}^{\infty} \frac{(-1)^{m}}{m} \left(\frac{c (p)}{p^{s}} + \frac{c (p^{2})}{p^{2 s}} + \ldots\right)^{m},
\end{split}
\end{equation*}
and hence
\begin{equation*}
c (p)
 = \left\{ \begin{array}{ll}
      \phi (q), & \mbox{if $p \equiv 1 \pmod{q}$;} \\
      0, & \mbox{if $p \not\equiv 1 \pmod{q}$.} \\
\end{array}
\right.
\end{equation*}

For all $n$ such that $c (n) \neq 0$, we have $n = A B$, where $A$ is coprime to $B$, $A$ is squareful, and $B$ is square-free, that is, $\mu (B) \neq 0$. Furthermore, all the prime factors of $B$ are congruent to 1 modulo $q$. Letting
\begin{equation*}
H(x)
 \mathrel{\mathop:}= \prod_{\substack{p \leq x, \ \!p\ \!\text{prime} \\ p \equiv 1 \pmod{q}}} p,
\end{equation*}
we have
\begin{equation*}
\begin{split}
\# \{n \leq x \colon c (n) \neq 0\}
 &\leq \# \{(A, B) \colon \text{$A$ squareful, $\mu (B) \neq 0$, $A B \leq x$, $B \mid H(x)$}\} \\
 &= \sum_{\substack{A \leq x \\ \text{$A$ squareful}}} \sum_{\substack{B \leq x / A \\ B \mid H(x)}} 1 \\
 &= \sum_{\substack{A \leq x \\ \text{$A$ squareful}}} \mathcal{B}\left(q,\frac{x}{A}\right) \\
 &= \sum_{\substack{A \leq \sqrt{x} \log x \\ \text{$A$ squareful}}} \mathcal{B}\left(q,\frac{x}{A}\right) + \sum_{\substack{\sqrt{x} \log x \leq A \leq x \\ \text{$A$ squareful}}} \mathcal{B}\left(q,\frac{x}{A}\right).
\end{split}
\end{equation*}
We examine the sums on the far right-hand side separately.

Using Lemma \ref{l1}, we see that
\begin{equation*}
\begin{split}
\sum_{\substack{A \leq \sqrt{x} \log x \\ \text{$A$ squareful}}}\mathcal{B}\left(q,\frac{x}{A}\right)
 &=O\left( \sum_{\substack{A \leq \sqrt{x} \log x \\ \text{$A$ squareful}}} \frac{x}{A}\left(\frac{\log\log x}{\log x}\right)^{1 -  1 / \phi (q)}\right) \\
 &= O \left(x\left(\frac{\log\log x}{\log x}\right)^{1 -  1 / \phi (q)} \sum_{\substack{A \leq \sqrt{x} \log x \\ \text{$A$ squareful}}} \frac{1}{A}\right) \\
 &= O \left(x\left(\frac{\log\log x}{\log x}\right)^{1 -  1 / \phi (q)} \sum_{a\geq1, b\geq 1} \frac{1}{a^2b^3}\right) \\
 &= O \left(x\left(\frac{\log\log x}{\log x}\right)^{1 -  1 / \phi (q)} \right).
\end{split}
\end{equation*}
Furthermore, we have
\begin{equation*}
\begin{split}
\sum_{\substack{\sqrt{x} \log x \leq A \leq x \\ \text{$A$ squareful}}} \mathcal{B}\left(q,\frac{x}{A}\right)
 &\leq \sum_{\substack{\sqrt{x} \log x \leq A \leq x \\ \text{$A$ squareful}}} \frac{x}{A} \\
 &\leq \sum_{\substack{\sqrt{x} \log x\leq A \leq x \\ \text{$A$ squareful}}} \frac{x}{\sqrt{x} \log x} \\
 &\leq \frac{\sqrt{x}}{\log x} \# \{A \leq x \colon \text{$A$ squareful}\} \\
 &= O \left(\frac{x}{\log x}\right).
\end{split}
\end{equation*}
	
Suppose that $\mathcal{P}_{1}, \ldots, \mathcal{P}_{r}$ are the prime ideals in the ring of integers of $K$ lying over the prime factors of $q$ and consider the Dirichlet series
\begin{equation*}
\sum_{n = 1}^{\infty} \frac{b (n)}{n^{s}}
 = \prod_{\mathcal{P} \mid p} \left(1 - \frac{1}{\|\mathcal{P}\|^{s}}\right)^{-1}.
\end{equation*}
For all $z$, we have
\begin{equation} \label{bnn}
\# \{n \leq z \colon b (n) \neq 0\}
 \leq \# \{\text{$n \leq z$ with all prime factors of $n$ in the sets $\mathcal{P}_{1}, \ldots, \mathcal{P}_{r}$}\}.
\end{equation}
It is well-known that the right-hand side of \eqref{bnn} is $O_{q} ((\log z)^{r})$. Thus, we have
\begin{equation*}
\# \{n \leq z \colon b (n) \neq 0\}
 = O_{q} ((\log z)^{r}).
\end{equation*}

For brevity's sake, we let
\begin{equation*}
\mathcal{A}
 = \{n \colon a (n) \neq 0\}, \quad
\mathcal{B}
 = \{m \colon b (m) \neq 0\}, \quad
\mathcal{C}
 = \{k \colon c (k) \neq 0\},
\end{equation*}
and denote
\begin{equation*}
\mathcal{A}_{\omega}
 = \mathcal{A} \cap [1, \omega], \quad
\mathcal{B}_{\omega}
 = \mathcal{B} \cap [1, \omega], \quad
\mathcal{C}_{\omega}
 = \mathcal{C} \cap [1, \omega].
\end{equation*}
Here, we note that
\begin{equation*}
\# \mathcal{B}_{\omega}
 = O_{r} ((\log \omega)^{r})
\end{equation*}
and
\begin{equation} \label{eq8}
\# \mathcal{C}_{\omega}
 = O_{q} \left(\omega\left(\frac{\log\log \omega}{\log \omega}\right)^{1 - 1 / \phi (q)}\right).
\end{equation}
Furthermore, we have
\begin{equation*}
\zeta_{K}(s)=\sum_{n \in \mathcal{A}} \frac{a (n)}{n^{s}}
 = \sum_{m \in \mathcal{B}} \frac{b (m)}{m^{s}} \sum_{k \in \mathcal{C}} \frac{c (k)}{k^{s}}.
\end{equation*}

On noting that $\mathcal{A} \subseteq \mathcal{B} \mathcal{C}$, where $\mathcal{B} \mathcal{C} = \{b c \colon b \in \mathcal{B}, c \in \mathcal{C}\}$, we have $\mathcal{A}_{x} \subset (\mathcal{B} \mathcal{C})_{x}$. It follows that
\begin{equation} \label{eq9}
\# \mathcal{A}_{x}
 \leq \# (\mathcal{B} \mathcal{C})_{x},
\end{equation}
where
\begin{equation} \label{eq10}
\# (\mathcal{B} \mathcal{C})_{x}
 = \sum_{\substack{b \leq x \\ b \in \mathcal{B}}} \sum_{\substack{c \leq x / b \\ c \in \mathcal{C}}} 1
 = \sum_{\substack{b \leq L \\ b \in \mathcal{B}}} \sum_{\substack{c \leq x / b \\ c \in \mathcal{C}}} 1 + \sum_{\substack{L < b \leq x \\ b \in \mathcal{B}}} \sum_{\substack{c \leq x / b \\ c \in \mathcal{C}}} 1,
\end{equation}
with $1 \leq L \leq x$ (to be chosen later). By \eqref{eq8}, we have
\begin{equation*}
\sum_{\substack{b \leq L \\ b \in \mathcal{B}}} \sum_{\substack{c \leq x / b \\ c \in \mathcal{C}}} 1
 \leq \sum_{\substack{b \leq L \\ b \in \mathcal{B}}} \# \mathcal{C}_{x / b}
 = O \left(\sum_{\substack{b \leq L \\ b \in \mathcal{B}}}  \frac{x}{b}\left(\frac{\log\log (x/b)}{\log (x/b)}\right)^{1 -  1 / \phi (q)}\right).
\end{equation*}
Since $b \leq L$, we have
\begin{equation*}
\left(\log \frac{x}{b}\right)^{1 - 1 / \phi (q)}
 > \left(\log \frac{x}{L}\right)^{1 - 1 / \phi (q)}.
\end{equation*}
Hence, we have
\begin{equation} \label{eq11}
\sum_{\substack{b \leq L \\ b \in \mathcal{B}}} \sum_{\substack{c \leq x / b \\ c \in \mathcal{C}}} 1
 = O \left(x\left(\frac{\log\log x}{\log x/L}\right)^{1 -  1 / \phi (q)} \sum_{\substack{b \leq L \\ b \in \mathcal{B}}} \frac{1}{b}\right)
 = O \left(x\left(\frac{\log\log x}{\log (x/L)}\right)^{1 -  1 / \phi (q)}\right),
\end{equation}
since
\begin{equation*}
\sum_{\substack{b\in\mathcal{B}}} \frac{1}{b}
 < \infty.
\end{equation*}

Next, we have 
\begin{equation} \label{eq12}
\sum_{\substack{L < b \leq x \\ b \in \mathcal{B}}} \sum_{\substack{c \leq x / b \\ c \in \mathcal{C}}} 1= \sum_{\substack{L < b \leq x \\ b \in \mathcal{B}}} \# \mathcal{C}_{x / b} \leq \sum_{\substack{L < b \leq x \\ b \in \mathcal{B}}} \frac{x}{b} \leq \frac{x}{L} \# \mathcal{B}_{x} = O\left(\frac{x (\log x)^{r}}{L}\right).
\end{equation}

In view of \eqref{eq9}, we substitute \eqref{eq11} and \eqref{eq12} into \eqref{eq10} to obtain
\begin{equation*}
\# \mathcal{A}_{x}
 = O \left(\frac{x (\log x)^{r}}{L}\right) + O \left(x\left(\frac{\log\log x}{\log (x/L)}\right)^{1 -  1 / \phi (q)}\right).
\end{equation*}
Then choosing $L = (\log x)^{r + 1}$, we obtain
\begin{equation*}
\# \mathcal{A}_{x}
 = O \left(x\left(\frac{\log\log x}{\log x}\right)^{1 - 1 / \phi (q)}\right).
\end{equation*}
This finishes the proof of Lemma \ref{l2}.
\end{proof}

\section{Proof of Proposition \ref{prop1}}

We show separately that $\abs{\zeta_{K, X} (s)} > 0$ in the right half-plane $\sigma \geq \beta$ and in the left-half plane $\sigma \leq \alpha$. More specifically, we want to find a $\beta$ so that 
\begin{equation*}
1 - \sum_{2 \leq n \leq  X } \frac{a (n)}{n^{\sigma}}
 > 0,
\end{equation*}
for $\sigma\geq \beta$. Toward this end, we employ the upper bound $a (n) \leq d (n)^{n_{0}-1}$, where $d (n)$ denotes the number of divisors of $n$ (see Chandrasekharan and Narasimhan \cite{ChandrasekharanNarasimhan1963}, Lemma 9) and satisfies the upper bound $d (n) \leq C_{\epsilon_{0}} n^{\epsilon_{0}}$ for all positive $\epsilon_{0}$ (see Hardy and Wright \cite{HardyWright2008}, Chapter XVIII, Theorem 317). Hence, we have $a (n) \leq C_{\epsilon_{0}, n_{0}} n^{\epsilon_{0} n_{0}}$.

It is enough to show that
\begin{equation} \label{eq13}
C_{\epsilon_{0}, n_{0}} \sum_{n = 2}^{\infty} \frac{1}{n^{\sigma - \epsilon_{0} n_{0}}}
 < 1.
\end{equation}
If we let $\epsilon_{0} < 1 / n_{0}$, then for $\sigma \geq \beta$ we have
\begin{equation*}
\sum_{n = 2}^{\infty} \frac{1}{n^{\sigma - \epsilon_{0} n_{0}}}
 \leq \sum_{n = 2}^{\infty} \frac{1}{n^{\beta - \epsilon_{0} n_{0}}}
 \leq \frac{1}{2^{\beta}} D_{\epsilon_{0}, n_{0}},
\end{equation*}
where
\begin{equation*}
D_{\epsilon_{0}, n_{0}}
 = \sum_{n = 2}^{\infty} \frac{4}{n^{2 - \epsilon_{0} n_{0}}}.
\end{equation*}
In order to obtain \eqref{eq13}, it is enough to have
\begin{equation*}
\beta
 > \frac{\log C_{\epsilon_{0}, n_{0}} D_{\epsilon_{0}, n_{0}}}{\log 2}.
\end{equation*}

We have
\begin{equation*}
\sum_{n = 2}^{\infty} \frac{d (n)^{n_{0}}}{n^{\beta}}
 \leq C_{\epsilon_{0}, n_{0}} \sum_{n =2}^{\infty} \frac{1}{n^{\beta - \epsilon_{0} n_{0}}}
 = \frac{1}{2^{\beta}} C_{\epsilon_{0}, n_{0}} D_{\epsilon_{0}, n_{0}}.
\end{equation*}
Then for $\sigma \geq \beta$, we have
\begin{equation} \label{eq14}
\left|\sum_{2 \leq n \leq X} \frac{a (n)}{n^{s}}\right|
\leq \sum_{2 \leq n \leq X} \frac{d (n)^{n_{0}}}{n^{\beta}}< 1,
\end{equation}
and hence
\begin{equation*}
\abs{\zeta_{K, X} (s)}
 \geq 1 - \left|\sum_{2 \leq n \leq  X } \frac{a (n)}{n^{s}}\right|
 > 0.
\end{equation*}
Therefore, $\zeta_{K, X} (s) \neq 0$ on the right-half plane $\sigma \geq \beta$.

Next, let $N$ be the largest positive integer less than or equal to $X$ for which the coefficient $a (N)$ is nonzero. Since
\begin{equation*}
\abs{\zeta_{K, X} (s)}
 \geq \frac{a (N)}{N^{\sigma}} - \left|\sum_{1 \leq n \leq N - 1} \frac{a (n)}{n^{s}}\right|,
\end{equation*}
it is enough to find an $\alpha$ such that 
\begin{equation*}
\frac{1}{N^{\sigma}}
 > \sum_{1 \leq n \leq N - 1} \frac{a (n)}{n^{\sigma}},
\end{equation*}
for $\sigma\leq \alpha$.

To this end, let us fix $\delta_{0}>0$. Then there exist constants $C_{\delta_{0}} > 0$ and $n_{\delta_{0}} \in \mathds{Z}^{+}$ such that for all $1 \leq n < n_{\delta_{0}}$, we have
\begin{equation*}
d (n)
 \leq C_{\delta_{0}} n^{(\delta_{0} + \log 2) / \log \log n},
\end{equation*}
and that for all $n \geq n_{\delta_{0}}$, we have
\begin{equation*}
d (n)
 \leq n^{(\delta_{0} + \log 2) / \log \log n}.
\end{equation*}
(See Wigert \cite{Wigert1906-1907}.)

It suffices to have
\begin{equation*}
\begin{split}
\frac{1}{N^{\sigma}}
 &> C_{\delta_{0}}^{n_{0}} \sum_{1 \leq n \leq n_{\delta_{0}} - 1} \frac{n^{(\delta_{0} + \log 2) n_{0} / \log \log n}}{n^{\sigma}} + \sum_{n_{\delta_{0}} \leq n \leq N - 1} \frac{n^{(\delta_{0} + \log 2) n_{0} / \log \log n}}{n^{\sigma}} \\
 &= 1 + C_{\delta_{0}}^{n_{0}} S_{I} (n_{0}, \delta_{0}, n_{\delta_{0}}, \sigma) + S_{II} (n_{0}, \delta_{0}, \sigma),
\end{split}
\end{equation*}
for $\sigma \leq \alpha$, where
\begin{equation*}
S_{I} (n_{0}, \delta_{0}, n_{\delta_{0}}, \sigma)
 = \sum_{2 \leq n \leq n_{\delta_{0}} - 1} \frac{n^{(\delta_{0} + \log 2) n_{0} / \log \log n}}{n^{\sigma}}
\end{equation*}
and
\begin{equation*}
S_{II} (n_{0}, \delta_{0}, \sigma)
 = \sum_{n_{\delta_{0}} \leq n \leq N - 1} \frac{n^{(\delta_{0} + \log 2) n_{0} / \log \log n}}{n^{\sigma}}.
\end{equation*}
This would follow from the inequality
\begin{equation*}
\frac{1}{N^{\alpha}}
 > 1 + C_{\delta_{0}}^{n_{0}} S_{I} (n_{0}, \delta_{0}, n_{\delta_{0}}, \alpha) + S_{II} (n_{0}, \delta_{0}, \alpha),
\end{equation*}
since, for any $\sigma \leq \alpha$,
\begin{equation*}
\begin{split}
\frac{1}{N^{\sigma}}
 &> \frac{1}{N^{\sigma - \alpha}} \left[1 + C_{\delta_{0}}^{n_{0}} S_{I} (n_{0}, \delta_{0}, n_{\delta_{0}}, \alpha) + S_{II} (n_{0}, \delta_{0}, \alpha)\right] \\
 &= \frac{1}{N^{\sigma - \alpha}} + C_{\delta_{0}}^{n_{0}} \sum_{2 \leq n \leq n_{\delta_{0}} - 1} \frac{n^{(\delta_{0} + \log 2) n_{0} / \log \log n}}{N^{\sigma - \alpha} n^{\alpha}} + \sum_{n_{\delta_{0}} \leq n \leq N - 1} \frac{n^{(\delta_{0} + \log 2) n_{0} / \log \log n}}{N^{\sigma - \alpha} n^{\alpha}}\\
 &> 1 + C_{\delta_{0}}^{n_{0}} \sum_{2 \leq n \leq n_{\delta_{0}} - 1} \frac{n^{(\delta_{0} + \log 2) n_{0} / \log \log n}}{n^{\sigma - \alpha} n^{\alpha}} + \sum_{2 \leq n \leq N - 1} \frac{n^{(\delta_{0} + \log 2) n_{0} / \log \log n}}{n^{\sigma - \alpha} n^{\alpha}} \\
 &= 1 + C_{\delta_{0}}^{n_{0}} S_{I} (n_{0}, \delta_{0}, n_{\delta_{0}}, \sigma) + S_{II} (n_{0}, \delta_{0}, \sigma).
\end{split}
\end{equation*}
Thus, it is enough to find $\alpha$ such that
\begin{equation} \label{eq15}
\frac{1}{N^{\alpha}}
 > 2 + 2 C_{\delta_{0}}^{n_{0}} S_{I} (n_{0}, \delta_{0}, n_{\delta_{0}}, \alpha)
\end{equation}
and such that
\begin{equation} \label{eq16}
\frac{1}{N^{\alpha}}
 > 2 S_{II} (n_{0}, \delta_{0}, \alpha).
\end{equation}

It is enough to have
\begin{equation} \label{eq17}
\frac{1}{N^{\alpha}}
 > 2 + 2 C_{\delta_{0}}^{n_{0}} \frac{1}{n_{\delta_{0}}^{\alpha}} \sum_{2 \leq n \leq n_{\delta_{0}} - 1} n^{(\delta_{0} + \log 2) n_{0} / \log \log n},
\end{equation}
since the right-hand side of \eqref{eq17} is greater than the right-hand side of \eqref{eq15}.

The inequality in \eqref{eq17} holds for any fixed $\alpha < 0$ and for all $N$ large enough in terms of $n_{0}$, $\delta_{0}$, $n_{\delta_{0}}$, $C_{\delta_{0}}$, and $\alpha$. Therefore, we may take any fixed $\alpha < 0$ as a function of $N$, $n_{0}$, and $\delta_{0}$ for which \eqref{eq16} holds true. 
For $n_{\delta_0}\geq 16$, we see that 
\begin{equation} \label{eq18}
\begin{split}
\sum_{n_{\delta_{0}} \leq n \leq N - 1} \frac{n^{(\delta_{0} + \log 2) n_{0} / \log \log n}}{n^{\alpha}}
 &\leq \sum_{n_{\delta_{0}} \leq n \leq N - 1} \frac{N^{(\delta_{0} + \log 2) n_{0} / \log \log N}}{n^{\alpha}} \\
 &< N^{(\delta_{0} + \log 2) n_{0} / \log \log N} \sum_{n_{\delta_{0} \leq n \leq N - 1}} \frac{1}{n^{\alpha}}.
\end{split}
\end{equation}
It remains to examine the sum on the far-right hand side of \eqref{eq18}.

For $\alpha < 0$, we have
\begin{equation*}
\sum_{n_{\delta_{0}} \leq n \leq N - 1} \frac{1}{n^{\alpha}}
 \leq (N - 1)^{-\alpha} + \int_{n_{\delta_{0}}}^{N - 1} \frac{\,dy}{y^{\alpha}}  
 < (N - 1)^{-\alpha} \left(\frac{ N-\alpha}{1-\alpha}\right).
\end{equation*}
It follows from \eqref{eq18} that \eqref{eq16} is consequence of 
\begin{equation*}
N^{-\alpha}
 > 2 N^{(\delta_{0} + \log 2) n_{0} / \log \log N} (N - 1)^{-\alpha} \left(\frac{ N-\alpha}{1-\alpha}\right).
\end{equation*}
One sees that an admissible choice of $\alpha$ is given by
\begin{equation*}
\alpha=-3(\delta_{0} + \log 2) n_{0} \frac{N \log N}{\log \log N}.
\end{equation*}
 Then $\zeta_{K, X} (s) \neq 0$ in the left-half plane $\sigma \leq \alpha$. This completes the proof of Proposition \ref{prop1}.

\section{Proof of Theorem \ref{thm1}}

Assuming for simplicity's sake that $T$ does not coincide with the ordinate of any zero, we have
\begin{equation*}
N_{K, X} (T)
  = \frac{1}{2 \pi i} \int_{R} \frac{\zeta_{K, X}' (s)}{\zeta_{K, X} (s)} \,ds,
\end{equation*}
where $R$ is the rectangle with vertices at $\alpha$, $\beta$, $\beta + i T$, and $\alpha + i T$.
Thus, we have
\begin{equation} \label{eq19}
2 \pi N_{K, X} (T)
 = \int_{R} \mbox{Im} \left(\frac{\zeta_{K, X}' (s)}{\zeta_{K, X} (s)}\right) \,ds
 = \triangle_{R} \arg \zeta_{K, X} (s),
\end{equation}
where $\triangle_{R}$ denotes the change in $\arg \zeta_{K, X} (s)$ as $s$ traverses $R$ in the positive sense.

Since $\zeta_{K, X} (s)$ is real and nonzero on $[\alpha,\beta]$, we have
\begin{equation} \label{eq20}
  \triangle_{[\alpha, \beta]} \arg \zeta_{K, X} (\sigma) \\
 = 0.
\end{equation}

As $s$ describes the right edge of $R$, we observe from \eqref{eq14} that 
\begin{equation*}
\abs{\zeta_{K, X} (s) -1}
 <1.
\end{equation*}
It follows that Re $\zeta_{K, X} (\beta + i t) > 0$ for $0 \leq t \leq T$. Hence, we have
\begin{equation} \label{eq21}
\triangle_{[0, T]} \arg \zeta_{K, X} (\beta + i t)
 = O (1).
\end{equation}

Furthermore, along the top edge of $R$, to estimate the change in $\arg \zeta_{K, X} (s)$ we decompose $\zeta_{K, X} (s)$ into its real part and its imaginary part. We have
\begin{equation*}
\zeta_{K, X} (s)= \sum_{n \leq \lbrack X \rbrack} a (n) \exp \{{-(\sigma + i t) \log n}\} = \sum_{n \leq \lbrack X \rbrack} \frac{a (n) [\cos (t \log n) - i \sin (t \log n)]}{n^{\sigma}},
\end{equation*}
so that
\begin{equation*}
\mbox{Im} (\zeta_{K, X} (\sigma + i T))
 = - \sum_{n \leq \lbrack X \rbrack} \frac{a (n) \sin (T \log n)}{n^{\sigma}}.
\end{equation*}
By a generalization of Descartes's Rule of Signs (see P\'{o}lya and Szeg\"{o} \cite{PolyaSzego1971}, Part V, Chapter 1, No. 77), the number of real zeros of $\mbox{Im} (\zeta_{K, X} (\sigma + i T))$ in the interval $\alpha \leq \sigma \leq \beta$ is less than or equal to the number of nonzero coefficients $a (n) \sin (T \log n)$. By Lemma \ref{l2}, the number of nonzero coefficients $a (n)$ is $O (X (\log\log X / (\log X)^{1 - 1 / \phi (q)})$ at most.

Since the change in argument of $\zeta_{K, X} (\sigma + i T)$ between two consecutive zeros of  $\mbox{Im} (\zeta_{K, X} (\sigma + i T))$ is at most $\pi$, it follows that
\begin{equation} \label{eq22}
  \triangle_{[\alpha,\beta]} \arg \zeta_{K, X} (\sigma + i T) \\
 = O \left(X\left(\frac{\log\log X}{\log X}\right)^{1 - 1 / \phi (q)}\right).
\end{equation}

As in the proof of Proposition \ref{prop1}, we let $N$ be the largest integer less than or equal to $X$ so that $a(N)\neq 0$. Along the left edge of $R$, we have
\begin{equation*}
\zeta_{K, X} (\alpha + i t)
 = \left[1 + \frac{1 + a (2) 2^{-\alpha - i t} + \ldots + a (N  - 1) ( N - 1)^{-\alpha - i t}}{a( N) N^{-\alpha - i t}}\right]a(N)N^{-\alpha - i t}.
\end{equation*}
Therefore, we have
\begin{equation} \label{eq23}
\begin{split}
  \triangle_{[0, T]} \arg \zeta_{K, X} (\alpha + i t) 
 &= \triangle_{[0, T]} \arg \left[1 + \frac{1 + a (2) 2^{-\alpha - i t} + \ldots + a (N  - 1) ( N - 1)^{-\alpha - i t}}{a( N) N^{-\alpha - i t}}\right] \\ & \quad + \triangle_{[0, T]} \arg a(N)N^{-\alpha - i t}.
\end{split}
\end{equation}

In the proof of Proposition \ref{prop1}, we noticed that
\begin{equation*}
\frac{a(N)}{N^{\alpha}}
 > \sum_{1 \leq n \leq N- 1} \frac{a (n)}{n^{\alpha}}.
\end{equation*}
Thus, for any $t$, we have
\begin{equation*}
\left\lvert \frac{1 + a (2) 2^{-\alpha - i t} + \ldots + a (N - 1) (N- 1)^{-\alpha - i t}}{a(N)N^{-\alpha - i t}}\right\rvert
 < 1,
\end{equation*}
and hence
\begin{equation} \label{eq24}
\triangle_{[0, T]} \arg \left[1 + \frac{1 + a (2) 2^{-\alpha - i t} + \ldots + a (N - 1) (N- 1)^{-\alpha - i t}}{a(N) N^{-\alpha -i t}}\right]
 = O (1).
\end{equation}

Finally, we have
\begin{equation} \label{eq25}
\begin{split}
\triangle_{[0, T]} \arg a (N) N^{-\alpha -i t}
&=\triangle_{[0, T]} \arg a (N) N^{-\alpha} \exp\{-i t \log N\}\\
&=\triangle_{[0, T]} \arg \exp\{-i t \log N\}\\
 &=- T \log N.
\end{split}
\end{equation}
Then substituting \eqref{eq24} and \eqref{eq25} into \eqref{eq23}, we obtain
\begin{equation} \label{eq26}
  \triangle_{[0, T]} \arg \zeta_{K, X} (\alpha + i t) = -T \log N + O (1).
\end{equation}
Since 
\begin{equation*}
\begin{split}
\triangle_{R} \arg \zeta_{K, X} (s)
 &= \triangle_{[\alpha, \beta]} \arg \zeta_{K, X} (\sigma) + \triangle_{[0, T]} \arg \zeta_{K, X} (\beta + i t)  \\ & \quad -\triangle_{[\alpha,\beta]} \arg \zeta_{K, X} (\sigma + i T)- \triangle_{[0, T]} \arg \zeta_{K, X} (\alpha + i t),
\end{split}
\end{equation*}
we may now substitute \eqref{eq20}, \eqref{eq21}, \eqref{eq22}, \eqref{eq26} into \eqref{eq19} to obtain Theorem \ref{thm1}.


\begin{thebibliography}{999}

\bibitem{BorweinFeeFergusonWaall2007} P. Borwein, G. Fee, R. Ferguson, and A. van der Waall, {\it Zeros of partial sums of the Riemann zeta function}, Experiment. Math. {\bf 16} (2007), no. 1, 21--40.

\bibitem{ChandrasekharanNarasimhan1963} K. Chandrasekharan and R. Narasimhan, {\it The approximate functional equation for a class of zeta-functions}, Math. Ann. {\bf 152} (1963) 30--64.

\bibitem{RammurtyCojocaru2006}M. R. Murty and A. C. Cojocaru, {\it An introduction to sieve methods and their applications}, London Mathematical Society Student Texts {\bf 66}, Cambridge University Press, Cambridge, 2006.

\bibitem{Davenport2000} H. Davenport, {\it Multiplicative number theory}, Graduate Studies in Mathematics {\bf 74}, Third edition (Edited by H. L. Montgomery), Springer-Verlag, New York, 2000.

\bibitem{GonekLedoan2010} S. M. Gonek and A. H. Ledoan, {\it Zeros of partial sums of the Riemann zeta-function}, Int. Math. Res. Not. {\bf 2010}, no. 10, 1775--1791.

\bibitem{HardyWright2008} G. H. Hardy and E. M. Wright, {\it An Introduction to the Theory of Numbers}, Sixth edition (Revised by D. R. Heath-Brown and J. H. Silverman, with a foreword by Andrew Wiles), Oxford University Press, Oxford, 2008.


\bibitem{HardyLittlewood} G. H. Hardy and J. E. Littlewood ,{\it The approximate functional equation for $\zeta(s)$ and $\zeta^2(s)$}, Proc. London Math. Soc. (2){\bf 29}, 81--97(1929).

\bibitem{KaratsubaVoronin1992} A. A. Karatsuba and S. M. Voronin, {\it The Riemann zeta-function}, De Gruyter Expositions in Mathematics {\bf 5} (Translated form the Russian by Neal Koblitz), Walter de Gruyter \& Co., Berlin, 1992.

\bibitem{Langer1931} R. E. Langer, {\it On the zeros of exponential sums and integral}, Bull. Amer. Math. Soc. {\bf 37} (1931), 213--239.

\bibitem{Montgomery1983} H. L. Montgomery, {\it Zeros of approximations to the zeta function}, Studies in Pure Mathematics, 497--506, Birkh\"{a}user, Basel, 1983.

\bibitem{MontgomeryVaughan2001} H. L. Montgomery and R. C. Vaughan, {\it Mean values of multiplicative functions}, Period. Math. Hungar. {\bf 43} (2001), 199--214.

\bibitem{Neukirch1999} J. Neukirch, {\it Algebraic number theory} (Translated from the 1992 German original and with a note by Norbert Schappacher. With a foreword by G. Harder), Grundlehren der Mathematischen Wissenschaften (A Series of Comprehensive Studies in Mathematics) {\bf 322}, Springer-Verlag, Berlin, 1999.

\bibitem{PolyaSzego1971} G. P\'{o}lya and G. Szeg\"{o}, \textit{Problems and theorems in analysis. Theory of functions, zeros, polynomials, determinants, number theory, geometry}, Volume II (Translated from the German by C. E. Billigheimer. Reprint of the 1976 English translation), Classics in Mathematics, Springer-Verlag, Berlin, 1998.

\bibitem{Titchmarsh1986} E. C. Titchmarsh, {\it The Theory of the Riemann zeta-function}, Second edition (Revised by D. R. Heath-Brown), The Clarendon Press, Oxford University Press, New York, 1986.

\bibitem{Wigert1906-1907} S. Wigert, {\it Sur l'order de grandeur du nombre des diviseurs d'un entier}, Ark. Mat. Astron. Fys. {\bf 3} (1906--1907), 1--9.
\end{thebibliography}
\end{document}